\documentclass[12pt]{amsart}
\usepackage{epsfig, color}
\usepackage{graphicx, multicol}
\usepackage{amssymb, latexsym, amsmath, amsthm, verbatim}
\usepackage{fullpage}
\newtheorem{theorem}{Theorem}
\newtheorem{lemma}[theorem]{Lemma}
\newtheorem{corollary}[theorem]{Corollary}

\begin{document}

\title{A Minimal Lamination with Cantor Set-Like Singularities}

\author{Stephen J. Kleene}
\address{Department of Mathematics\\
MIT \\
77 N. Masachussetts Ave.  \\
Boston, MA 2139}

\email{skleene@math.mit.edu}

\maketitle

\begin{abstract}
Given a compact closed subset $M$ of a line segment in $\mathbb{R}^3$, we construct a sequence of  minimal surfaces $\Sigma_k$ embedded in a neighborhood $C$ of the line segment that converge smoothly to a limit lamination of $C$ away from $M$. Moreover, the curvature of this sequence blows up precisely on $M$, and the limit lamination has non-removable singularities precisely on the boundary of $M$.
\end {abstract}
\section{introduction}
Let $\Sigma_k \subset B_{R_k}  = B_{R_k} (0) \subset \mathbb{R}^3$ be a sequence of compact embedded minimal surfaces with $\partial \Sigma_k \subset \partial B_{R_k}$ and curvature blowing up at the origin. In \cite{CM1}, Colding and Minicozzi showed that when $R_k \rightarrow \infty$, a subsequence  converges off a Lipshitz curve to a foliation by parallel planes. In particular, the limit is a smooth, proper foliation.  By contrast, in \cite{CM2} Colding and Minicozzi constructed  a sequence as above with $R_k$ uniformly bounded and converging to a limit lamination of the unit ball with a non-removable singularity at the origin. Later, B. Dean in \cite{BD} found a similar example where the limit lamination has a finite set of  singularities along a line segment, and S. Khan in \cite{SK} found a limit lamination consisting of a non-properly embedded minimal disk in the upper half ball spiraling into a foliation by parallel planes of the lower half ball.  Both Dean and Khan used methods that are analogous to those in \cite{CM1}. Recently, using a variational method, D. Hoffman and B. White in \cite{HW} were able to construct a sequence converging to a non-proper limit lamination and with curvature blowup occurring along an arbitrary compact subset of a line segment. In this paper we do the same, but with a method that is derivative of that in \cite{CM1} and \cite{SK}. The main theorem is: 

\begin{theorem}  \label{main_theorem}
Let $M$ be a compact subset of $\{ x_1 = x_2 = 0, |x_3| \leq 1/2 \}$ and let $C = \{ x_1^2  + x_2^2 \leq 1, |x_3| \leq 1/2 \}$. Then there is a sequence of properly embedded  minimal disks $\Sigma_k \subset C$ with $\partial \Sigma_k \subset \partial C$ and containing the vertical segment $\{ (0, 0, t)| |t| \leq 1/2\}$ so that:

\begin{enumerate} 
\item [(A)] $\lim_{k \rightarrow \infty} |A_{\Sigma_k}|^2 (p) = \infty$ for all $p \in M$. \label{curvature_blowup}
\item [(B)]For any $\delta > 0$ it holds $\sup_k \sup_{\Sigma_k \setminus M_\delta} |A_{\Sigma_k}|^2 < \infty$ where $M_\delta =  \cup_{p \in M} B_\delta(p)$. \label{curvature_stays_bounded}
\item [(C)]$\Sigma_k \setminus \{ x_3-\text{axis}\} = \Sigma_{1,k} \cup \Sigma_{2, k}$ for multi-vauled graphs $\Sigma_{1,k}$, $\Sigma_{2,k}$. \label{m_valued_graphs}
\item[(D)] For each interval $ I = (t_1, t_2)$ of the compliment of $M$ in the $x_3$-axis, $\Sigma_k \cap \{ t_1 <x_3 < t_2\}$ converges to an imbedded minimal disk $\Sigma_I$ with $\bar{\Sigma_I} \setminus \Sigma_I = C \cap \{ x_3 = t_1, t_2\}$. Moreover, $\Sigma_I \setminus \{ x_3-\text{axis}\} = \Sigma_{1,I} \cup \Sigma_{2,I}$, for $\infty$-valued graphs $\Sigma_{1,I}$ and $\Sigma_{2, I}$ each of which spirals in to the planes $\{x_3 = t_1\}$ from above and $\{ x_3 = t_2\}$ from below.  \label {limit_lamination}
\end{enumerate}
\end{theorem}
It follows from (D) that a subsequence of the $\Sigma_k \setminus M$ converge to a limit lamination of $C \setminus M$. The leaves of this lamination are given by the multi-valued graphs $\Sigma_I$ given in (D), indexed by intervals $I$ of the complement of $M$, taken together with the planes $\{ x_3 = t\} \cap C$ for $(0,0,t) \in M$. This lamination does not extend to a lamination of $C$, however, as every boundary point of $M$ is a non-removable singularity. Theorem \ref{main_theorem} is inspired by the result of Hoffman and White in \cite{HW}. The author would like to thank Professor William P. Minicozzi who suggested to him the problem of constructing an example with singularities on the  Cantor set (a special case of Theorem \ref{main_theorem})

Throughout we will use coordinates $(x_1, x_2, x_3)$ for vectors in $\mathbb{R}^3$, and $z = x + iy$ on $\mathbb{C}$. For $p \in \mathbb{R}^3$, and $s > 0$, the ball in $\mathbb{R}^3$ is $B_s(p)$.  We denote the sectional curvature of a smooth surface $\Sigma$ by $K_\Sigma$. When $\Sigma$ is immersed in $\mathbb{R}^3$, $A_\Sigma$ will be its second fundamental form. In particular, for $\Sigma$ minimal we have that $|A_\Sigma|^2 = -2K_\Sigma$. Also, we will  identify the set $M \subset \{ x_3 \text{-axis} \}$ with the corresponding subset of $\mathbb{R} \subset \mathbb{C}$; that is, the notation will not reflect the distinction, but will be clear from context. Our example will rely heavily on the Weierstrass Representation, which we introduce here.

\section{The Weierstrass representation}

Given a domain $\Omega \subset \mathbb{C}$, a meromorphic function $g$ on $\Omega$ and a holomorphic one-form $\phi$ on $\Omega$, one obtains a (branched) conformal minimal immersion $F: \Omega \rightarrow \mathbb{R}^3$, given by (c.f. \cite{Os})

\begin{equation} \label{W_R}
F(z) = Re \left\{ \int_{\zeta \in  \gamma_{z_0, z}} \left( \frac{1}{2}\left(g^{-1}(\zeta) - g(\zeta) \right), \frac{i}{2} \left( g^{-1}(\zeta) + g(\zeta)\right), 1\right) \phi (\zeta) \right\}
\end{equation}
the so-called Weierstrass representation associated to $\Omega, g, \phi$. The triple ($\Omega$, $g$, $\phi$) is referred to as the Weierstrass data of the immersion $F$. Here, $\gamma_{z_0, z}$ is a path in $\Omega$ connecting $z_0$ and $z$. By requiring that the domain $\Omega$ be simply connected, and that $g$ be a non-vanishing holomorphic function, we can ensure that $F(z)$ does not depend on the choice of path from $z_0$  to $z$, and that $dF \neq 0$. Changing the base point $z_0$ has the effect of translating the immersion by a fixed vector in $\mathbb{R}^3$.

The unit normal $\bold{n}$ and the Gauss curvature $K$ of the resulting surface are then (see sections 8, 9 in \cite{Os})

\begin{eqnarray}
\bold{n} & = & \left( 2 \text{Re} \, g, 2 \text{Im} \, g, |g|^2 - 1 \right)/ \left( |g|^2 + 1\right), \label{unit_normal} \\
K & = & - \left[ \frac{4 |\partial_z g| |g|}{|\phi|(1 + |g|^2)^2}\right]^2. \label{gauss_curvature}
\end{eqnarray}
Since the pullback $F^*(dx_3)$ is $\text{Re} \,  \phi$, $\phi$ is usually called the \textit{height differential}. The two standard examples are 

\begin{equation}
g(z) = z , \phi(z) = dz/z, \Omega = \mathbb{C} \setminus \{0\} 
\end{equation}
giving a catenoid, and 

\begin{equation}
g(z) = e^{iz}, \phi(z) = dz, \Omega = \mathbb{C}
\end{equation}
giving a helicoid.
\begin{figure}[htbp]
\centering
\begin{minipage}[b]{1 \textwidth}
\begin{center}
\includegraphics[width = .8 \textwidth]{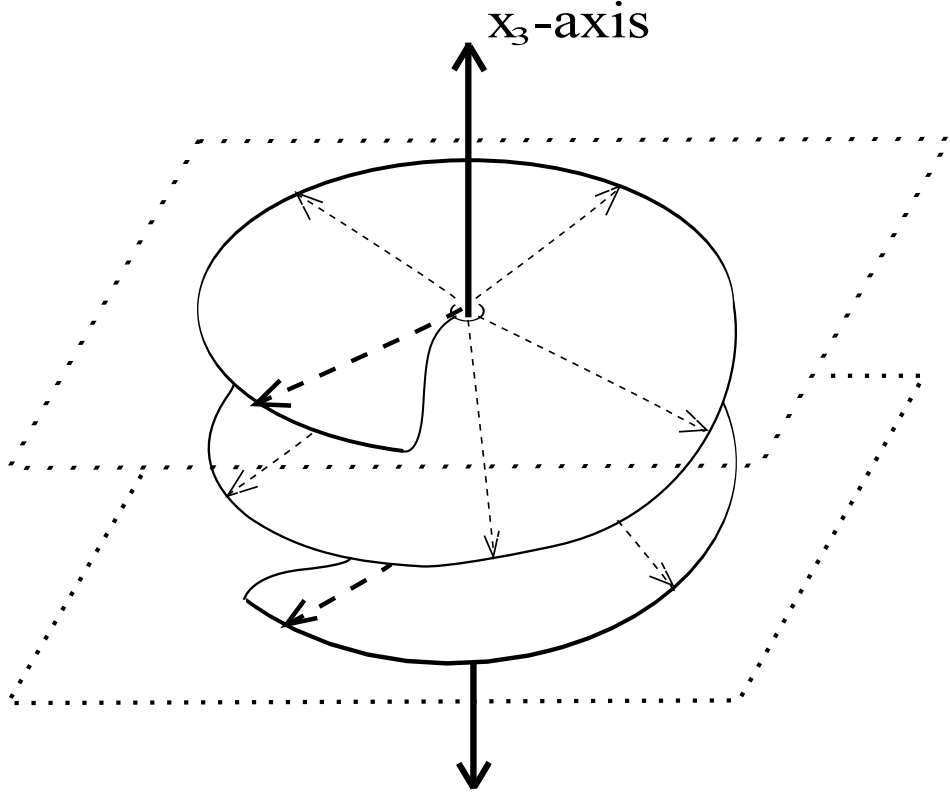}
\caption{A schematic of a helicoid like apparently 3-valued graph  spiraling between two horizontal planes. (Image due to Siddique Khan)}
\end{center}
\end{minipage}
\end{figure}

  We will always write our non-vanishing holomorphic function $g$ in the form $g = e^{ih}$, for a potentially vanishing holomorphic function $h$, and we will always take $\phi = dz$. For such Weierstrass data, the differential $dF$ may be expressed as 

\begin{eqnarray} \label{d_xF}
\partial_x F & = &\left( \text{sinh} \, v \, \text{cos}  \, u, \text{sinh}  \, v \text{sin}\, u, 1\right), \\
\partial_y F & = & \left( \text{cosh} \, v  \, \text{sin} \, u, - \text{cosh} \, v \, \text{cos} \, u, 0\right).\label{d_yF}
\end{eqnarray}

\section{An outline of the proof}
Fix a compact subset $M$ of the real line. We will be dealing with a family of immersions $F_{k,a}: \Omega_{k, a} \rightarrow \mathbb{R}^3$ that depend on a parameter $ 0 < a < 1/2$ given by Weierstrass data of the form $\Omega_{k,a}, G_{k,a} = e^{iH_{k,a}},  \phi = dz$, and a  sequence $M_k \subset M$ that converge to a dense subset of $M$. Each function $H_{k,a}$ will be real valued when restricted to the real line in $\mathbb{C}$. That is, writing $H_{k,a} = U_{k,a} + iV_{k,a}$ for real valued functions $U_{k,a}, V_{k,a}: \Omega_k \rightarrow \mathbb{R}$, we have that $H_{k,a}(x, 0) = U_{k,a}(x,0)$. Moreover, we will show that $V_{k,a} (x,y) > 0 $ when $y > 0$. A look at  the expression for the unit normal given above in (\ref{unit_normal}) then shows that all of the surfaces $\Sigma_{k,a} := F_{k, a}(\Omega_{k, a})$ will be multi-valued graphs over the $(x_1,x_2)$ plane away from the $x_3$-axis (since $|g(x,y)| = 1$ is equivalent to $y = 0$). The dependence on the  parameter $0< a < 1/2$ will be such that $\text{lim}_{a \rightarrow 0} |A_{\Sigma_{k,a}}|^2(p) = \infty$ for all $p \in M_k$, and such that $|A_{\Sigma_{k,a}}|^2$ remains uniformly bounded in $k$ and $a$ away from $M$. We will then choose a suitable sequence $a_k  \rightarrow 0$, and set $F_k = F_{k, a_k}$, $\Omega_k = \Omega_{k, a_k}$, $G_k = G_{k, a_k}$, and $H_k = H_{k, a_k}$.  Immediately, (A), (B) and (C) of Theorem \ref{main_theorem} are satisfied by the diagonal subsequence.  In fact,  we will show that any suitable sequence is  a sequence $a_k \rightarrow 0$ satisfying $a_k < \gamma^{-k}$ for a parameter $\gamma > 1$ which we introduce later.  The bulk of the work will go towards establishing (D). To this end, we will show that

\begin{lemma} \label{main_lemma} 

\begin{enumerate}
\item [(a)]The horizontal slice $\{ x_3 = t\} \cap F_k (\Omega_k)$ is the image of the vertical segment $\{ x = t\}$ in the plane, i.e., $ x_3 (F_k (t,y)) = t$. \label{line_to_plane}

\item [(b)]The image of $ F_k(\{ x = t\})$ is a graph over a line segment in the plane $\{ x_3 = t\} $(the line segment will depend on t) \label{graph}

\item [(c)]The boundary of the graph in (b) is outside the ball $B_{r_0}(F_k(t, 0))$ for some $r_0 > 0$. \label{long_enough}
\end{enumerate}
\end{lemma}
This gives that the immersions $F_k: \Omega_k \rightarrow \mathbb{R}^3$ are actually embeddings, and that the surfaces $\Sigma_k$ given by $F_k(\Omega_k)$ are all embedded in a fixed cylinder $C_{r_0} = \{ x_1^2 + x_2^2 \leq r_0^2, |x_3| < 1/2\}$ about the $x_3$-axis in $\mathbb{R}^3$. This will then imply that the surfaces $\Sigma_k$ converge smoothly on compact subsets of $C_{r_0} \setminus M$ to a limit lamination of $C_{r_0}$. The claimed structure of the limit lamination(that is, that on each interval of the complement it consists of two multi-valued graphs that spiral into planes from above and below) will be established at the end. 

\begin{figure}[htbp]
\centering
\begin{minipage}[b]{1 \textwidth}
\begin{center}
\includegraphics[width = .8 \textwidth]{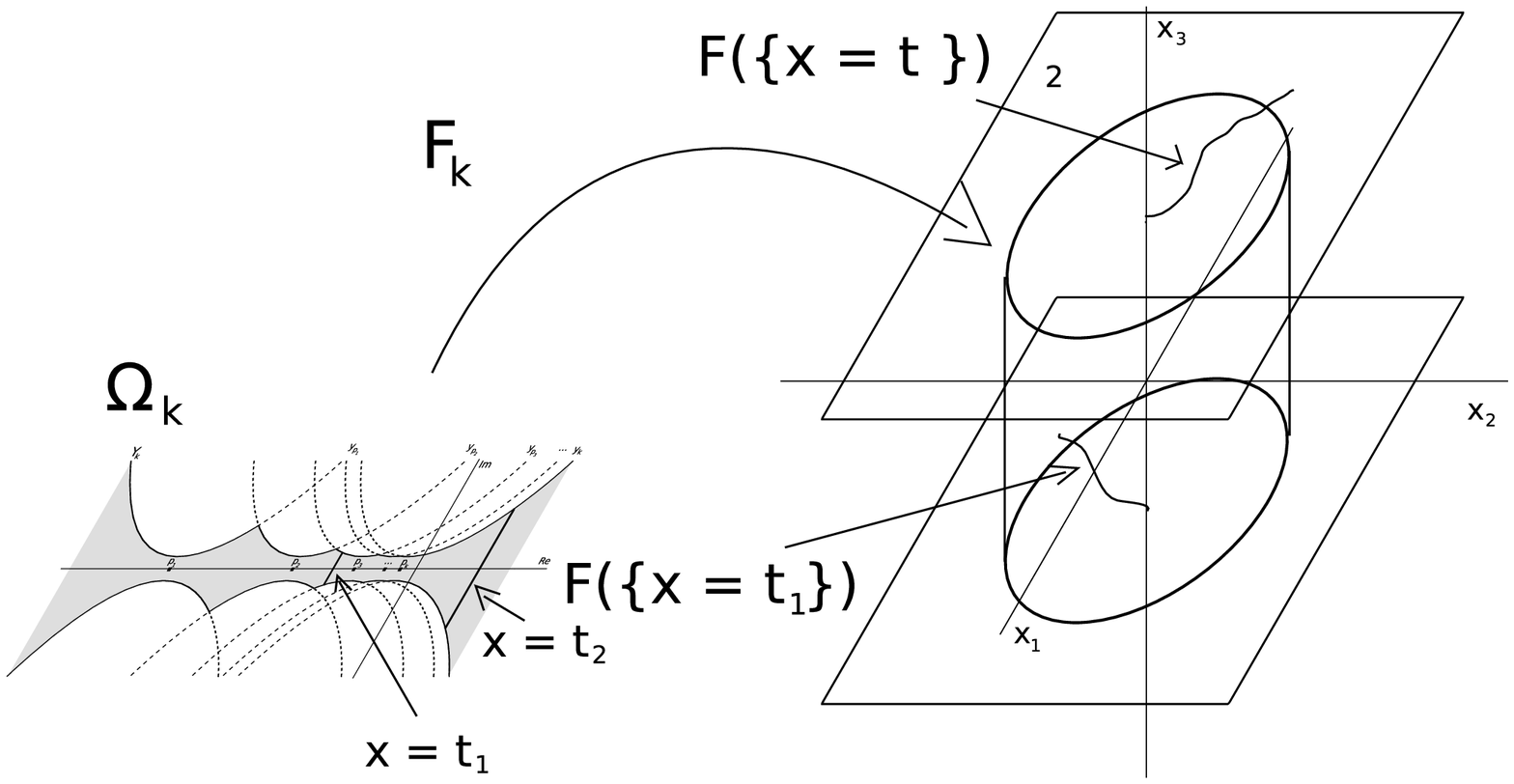}
\caption{The functions $F_k$ map vertical rays of the form $\{x = t\}$ contained in the domain $\Omega_k$ to planes perpendicular to the $x_3$-axis given by $\{ x_3 = t \}$. Note that this induces an identification of the closed set $M$, thought of as lying in the complex plane along the real axis, with its image in the $x_3$-axis.}
\end{center}
\end{minipage}
\end{figure}

Throughout the paper, all computations will be carried out and recorded only on the upper half plane in $\mathbb{C}$, as the corresponding computations on the lower half plane are completely analogous. By scaling it suffices to prove Theorem \ref{main_theorem} (D) for some $C_{r_0}$, not $C_1$ in particular.

\section{definitions}

Let $ M \subset [0,1]$ be a closed set. Fix $\gamma > 1$, and take $M_{-1} $ to be the empty set. Then for $k$ a non-negative integer, we inductively define two families of sets $m_k$, and $M_k$ as follows: Assuming  $M_{k-1}$ is already defined, take $m_k$ to be any maximal subset of $M$ with the property that, for $p,q \in m_k$,  $r \in M_{k-1}$, it holds that $|p - q|, |p - r| \geq \gamma^{-k}$. Then define $M_k$ = $M_{k - 1} \cup m_k$ and $M_\infty = \cup_k M_k$. Also, for $x \in \mathbb{R}$ define $p_k(x)$ to be the closest element in $M_k$ to $x$. Note that there are at most two such points, and we take $p_k(x)$ to be the closest point on the left, equivalently the smaller of the two points. For $p \in M_\infty$, we define $e(p)$ to be the unique natural number such that $p \in m_{e(p)}$. Note that $e(p_k(x)) \leq k$. We take

\begin{equation} \label{h_def}
h_a(z) = \int_0^z \frac{dz}{\left( z^2 + a^2 \right)^2} = u_a(z) + iv_a(z)
\end{equation}
and
\begin{equation} \notag
y_{0,a}(x) = \epsilon   \left(x^2 + a^2 \right)^{5/4}
\end{equation}
for $\epsilon$ to be determined. For $p \in \mathbb{R}$ we define
\begin{equation} \notag
h_{p,a}(z) = h_a(z - p) = u_{p,a}(z) + iv_{p,a}(z)
\end{equation}
and 
\begin{equation} \notag
y_{p,a}(x) = y_{0,a}(x - p).
\end{equation} 
We then take
\begin{equation} \label{h_l}
h_{l,a}(z)  = \sum_{p \in m_l} h_{p,a}(z) = u_{l,a}(z) + iv_{l,a}(z)
\end{equation}
and
\begin{equation} \notag
y_{l,a}(x) = \min_{p \in m_l} y_{p,a}(x).
\end{equation}
We take
\begin{equation} \label{H_k}
H_k(z) = \sum_{l = 0}^k \mu^{-l} h_{l, a_k}(z) = U_k(z) + iV_k(z)
\end{equation}
for a parameter $\mu > \gamma$ to be determined. We take 
\begin{equation} \notag
Y_k(x)  = \min_{l \leq k} y_{l, a_k}(x).
\end{equation}

\begin{figure}[h]
\begin{minipage}[b]{1 \textwidth}
\begin{center}
\includegraphics[width = .8 \textwidth]{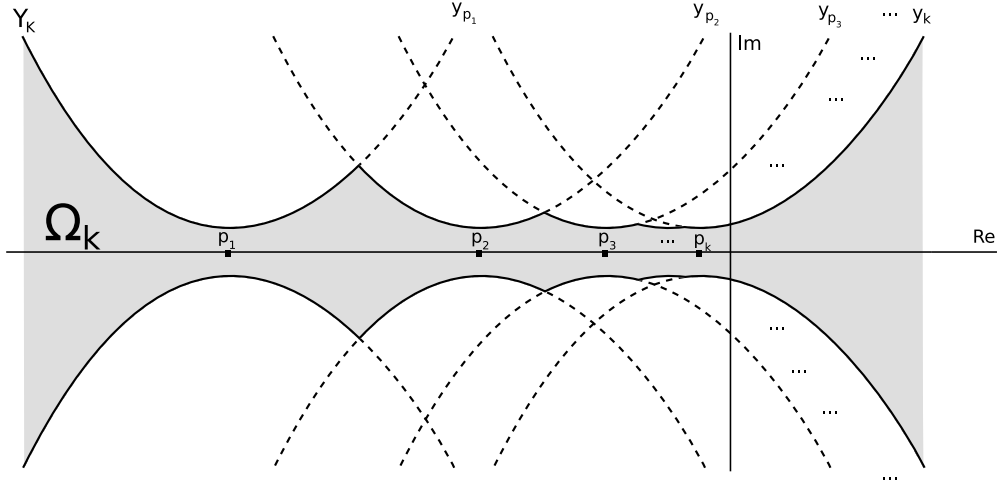}
\caption{A schematic rendering of the domain $\Omega_k$ in the case of   $M = \{ p_l = - 2^{-l} | l \in \mathbb{N} \}$. The solid line indicates the function $Y_k(x)$, and the shaded region indicates the domain $\Omega_k$ itself. Note that in this case, the sets $m_l = \{ p_l\}$ consist of a single point}
\end{center}
\end{minipage}
\end{figure}
We take
\begin{equation} \notag
\omega_a = \left\{ x + iy | |y| \leq y_{0,a}(x) \right\}, \omega_{p,a} =  \left\{ x + iy | |y| \leq y_{p,a}(x) \right\}
\end{equation}
and 
\begin{equation} \notag
\omega_{l,a} =  \left\{ x + iy | |y| \leq y_{l,a}(x) \right\}, \Omega_k =  \left\{ x + iy | |y| \leq Y_k(x) \right\}
\end{equation}
and lastly set $\Omega_\infty = \cap_k \Omega_k$.

Note that in the above definitions, objects bearing the subscript ``$k$'' (as opposed to ``$l$'') always enumerate an (as yet undetermined) diagonal sequence. Consequently, the dependence on the parameter $a$ is omitted from the notation. At times, the dependence on $a$ will be suppressed from the notation for objects without the subscript ``$k$''. Also, note that  for each $x$ we have that  $Y_k(x) = y_{p_k, a_k}(x)$. Again, when it is clear, the subscript ``$a_k$'' will be suppressed. Keep in mind throughout that $\{ a_k \}$ will always denote a sequence with $a_k \leq \gamma^{-k}$. Also, the parameters $\gamma$ and $\mu$ introduced in this section must satisfy $\mu^{2/3} < \gamma < \mu < \gamma^3$. The reasons are technical, and should become clear later in the paper. 
\section{Preliminary Results} \label{preliminary_results}

We record some elementary  properties of the sets $M_k$ and $m_k$ defined above which will be needed later.

\begin{lemma} 
$|m_k| \leq \gamma^k + 1$
\end{lemma}

\begin{proof}
Let $p_1 < \ldots < p_n$ be $n$ distinct elements of $m_k$, ordered least to greatest. By construction we have that $p_{k + 1} - p_k \geq \gamma^{-k}$. Also, since $p_1, p_n \in M$ we get

\begin{equation} \notag
1  \geq p_n - p_1 = \sum_{k = 1}^{n - 1} p_{k + 1} - p_k \geq (n-1)\gamma^{-k}
\end{equation} 
\end{proof} 

\begin{lemma}
For all $p$ in M, there is  a $q $ in  $M_k$ such that $|p - q| < \gamma^{-k}$
\end{lemma}
\begin{proof}
If not, $m_k$ is not maximal.
\end{proof}

\begin{lemma}
The union $\cup_{k = 0}^{\infty} m_k = \cup_{k = 0}^{\infty} M_k \equiv M^\infty$ is a dense subset of $M$
\end{lemma}

\begin{proof}
Suppose not. Then there is a $q \in M$, and a positive integer $k$ such that  $|p -q| > \gamma^{-k}$, $\forall p \in M^\infty$. In particular, this implies that $m_k$ is not maximal. 
\end{proof}

In order to avoid disrupting the narrative, the proofs of the remaining results in this section will be recorded later in the Appendix at the end. The proofs are somewhat tedious, though easily verified.
\begin{lemma} \label{holomorphic_on_domain}. 

For $\epsilon$ sufficiently small,  $h_p(z)$ is holomorphic on $\omega_p$, $h_l$ is holomorphic on $\omega_l$, and $H_k$ is holomorphic on $\Omega_k$.
\end{lemma}
We will also need the following estimates:
\begin{lemma} \label{uv_bounds}
 On the domain $\omega_p$ it holds that:
\begin{equation} \notag
\left|\frac{\partial}{\partial y} u_p(x,y)\right| \leq \frac{c_1|x- p||y|}{((x-p)^2 + a^2)^3}
\end{equation}
and
\begin{equation} \notag
\frac{\partial}{\partial y} v_p(x,y) > \frac{c_2}{((x -p)^2 + a^2)^2}.
\end{equation}
\end{lemma}

Integrating the above estimates from $0$ to the upper boundary of $\omega_p$ gives
\begin{equation} \notag
\big| u_p(x,y_p(x)) - u_p(x,0) \big| \leq \epsilon^2 c_1
\end{equation}
and
\begin{equation} \notag
\text{min}_{[y_p(x)/2, y_p(x)]} v_p(x,y) > \frac{ \epsilon c_2}{2((x-p)^2 + a^2)^{3/4}}
\end{equation}
These estimates immediately give 
\begin{corollary} \label{U_k_V_k_bounds} We have the bounds
\begin{equation}  \label{U_k_bound}
\left| U_k(x, Y_k(x)) - U_k(x, 0) \right| \leq \epsilon^2 c_1\left\{ \sum_{l = 0}^k \left(\gamma /\mu \right)^l  + \mu^{-l} \right\} \leq \epsilon^2 c_1'
\end{equation}
and
\begin{eqnarray} \label{V_k_bound} 
V_k(x, Y_k(x)/2) &\geq & \frac{ \epsilon c_2 }{2}\sum_{l = 0}^k \mu^{-l} \sum_{p \in m_l} \frac{Y_k(x)}{y_p(x)} ((x - p)^2 + a_k^2)^{- 3/4} \\ \notag
& \geq & y_{p_k, a_k}(x) \frac{ \epsilon c_2 }{2} \sum_{l = 0}^k \mu^{-e(p_l(x))}  \frac{1}{y_{p_l, a_k}(x)} ((x - p_l(x))^2 + a_k^2)^{- 3/4} \\ \notag
& = &  \frac{ \epsilon c_2 }{2} q_k(x)
\end{eqnarray}
where $q_k(x)$ is defined by the last equality above. 
\end{corollary}

\section{Proof of Lemma \ref{main_lemma}}

We will first concern ourselves with establishing Lemma \ref{main_lemma}. (a) follows from (\ref{W_R}) and the choice of $z_0 = 0$.  Choosing $\epsilon < \epsilon_0 < c_1'^{-1/2}$, where $c_1'$ is the constant in (\ref{U_k_bound}), and using (\ref{d_yF}) we get
\begin{equation} \label{angle_bound}
\langle \partial_y F_k(x,y), \partial_y F_k(x,0) \rangle = \text{cosh} \, V_k(x,y) \, \text{cos}(U_k(x,y_0(x)) - U_k(x, 0)) > \text{cosh} \, V_k (x,y)/2
\end{equation}
Here we have used that $\text{cos}(1) > 1/2$. This gives  that all of the maps $F_k: \Omega_k \rightarrow \mathbb{R}^3$ are indeed embeddings (for all values of $a$) and proves (b) of Lemma \ref{main_lemma}.

Now, integrating (\ref{angle_bound}) from $Y_k(x)/2$ to $Y_k(x)$ gives 
 \begin{equation}
 \langle F_k (x,Y_k(x)) - F_k(x, 0), \partial_y F_k(x, 0)\rangle > \frac{Y_k(x)}{2}e^{\text{min}_{[Y_k/2, Y_k]} V_k}
 \end{equation}
 Using the bound for $V_k$ recorded in (\ref{V_k_bound}), we get that
 \begin{equation}
 \langle F_k (x,Y_k(x)) - F_k(x, 0), \partial_y F_k(x, 0)\rangle > \frac{\epsilon}{2}s_k^{5/3} e^{\epsilon \frac{ c_2}{2}q_k(x)} 
 \end{equation}
with $s_k(x) = ((x - p_k(x))^2 + a_k^2)^{3/4}$. Take $r_k(x) \equiv \frac{\epsilon}{2} s_k^{5/3} e^{\frac{1}{2} \epsilon c_2 q_k(x)}$. We will show that   $r_k(x)$ remains uniformly large in $k$; this establishes (c) of Lemma \ref{main_lemma}. First, we need  Lemmas (\ref{alpha_delta}) and (\ref{good_interval}) below. In the following, take $\Phi(\xi) = \xi^{5/3} e^{\frac{1}{2} _2 c_2 \epsilon \xi^{-1}}$.

\begin{lemma} \label{alpha_delta}
For all $ \alpha > 0$, there exists a $\delta = \delta(\alpha)$ such that
\begin{equation} \label{better_than_polynomial} 
\Phi(\xi) \geq \xi^{-\alpha}
\end{equation} 
for $0 < \xi < \delta$.
\end{lemma}

\begin{proof}
\begin{equation} \notag
\lim_{\xi \rightarrow 0} \frac{\epsilon}{2} \xi^{5/3 + \alpha} e^{ \frac{1}{2}\epsilon c_2 \xi^{-1}} = \infty
\end{equation}
for all $\alpha$.
\end{proof}
We now choose $\mu$ and $\sigma$ so that $\mu^{2/3} < \mu^{(1 + \sigma)2/3} < \gamma < \mu < \gamma^3$.  We must also choose $\alpha$ so that $\alpha \sigma - 5/3 \geq 0$, as will be seen in the following. In fact, for later applications, we demand $\alpha \sigma - 5/3 \geq 1$.
\begin{lemma} \label {good_interval}
For $|x - p_k|, a_k \leq \mu^{-2/3(1 + \sigma)k} \left( \frac{\delta(\alpha)^{2/3}}{\sqrt{2}} \right)$, we have that
\begin{equation} \notag
r_k(x) > 1.
\end{equation}
\end{lemma}

\begin{proof}
The assumptions immediately give that 

\begin{equation} \notag
s_k(x) = ((x -p_k)^2 + a_k^2)^{3/4} < \mu^{-(1 + \sigma)k} \delta < \delta
\end{equation}
which we rewrite as 

\begin{equation} \notag
\mu^k s_k \leq \mu^{-\sigma k} \delta.
\end{equation}
Applying (\ref{better_than_polynomial}) and using that $e(p_k(x)) \leq k$ we find that 

\begin{equation} \notag
\Phi(\mu^{e(p_k)} s_k) > \left( \mu^{- \sigma k} \delta \right)^{-\alpha}.
\end{equation}
Equivalently, 
\begin{equation} \notag
 r_k(x) \geq \frac{\epsilon}{2} s_k^{5/3} e^{\frac{\epsilon}{2} c_2 \mu^{-e(p_k)} s_k^{-1}} > \mu^{(\alpha \sigma  - 5/3) k} \delta^{- \alpha} > 1
\end{equation}
since we have chosen $\alpha \sigma - 5/3 \geq 1$, and we may assume $\delta < 1$.
\end{proof}

We are ready to prove:

\begin{lemma} (Lemma \ref{main_lemma} (c))\label{r_k_bound_lemma} There exists a sequence $\{c_k\}$ with $c_k > 0$ and $\prod_{l = 0}^\infty c_l > 0$ such that if $r_k(x) < 1$, then
\begin{equation} \notag
r_k (x) > c_k r_{k-1} (x).
\end{equation}
\end{lemma}

\begin{proof}
Recall that $Y_k(x) = y_{p_k}(x)$ and $Y_{k - 1}(x) = y_{p_{k - 1}}(x)$. If $|x - p_k| < \mu^{-2/3(1 + \sigma)k} \delta^{2/3}/ \sqrt{2}$, then
\begin{equation} \notag
r_k(x) > 1
\end{equation}
by Lemma \ref{good_interval}. So we assume that $|x - p_k|> c_0 \mu^{-2/3(1 + \sigma)k} $ with $c_0 = \delta^{2/3}/ \sqrt{2} $. By the construction of the sets $m_k, M_k$, we have that  $|p_k - p_{k - 1}|  < \gamma^{-k +1}$ We also have that $|p_{k - 1}(x) - x| > c_0 \mu^{-2/3(1 + \sigma)k}$.
Then we may estimate
\begin{eqnarray}
\left[ \frac{y_{p_{k - 1}, a_k}(x)}{y_{p_{k - 1}, a_{k- 1}}(x)} \right]^{4/5} &= & \frac{((x - p_{k - 1})^2 + a_k^2)}{((x - p_{k - 1})^2 + a_{k-1}^2)}  > \frac{1}{1 + c_0^{-2}\gamma^{- 2} \tau^{2{k- 1}}}. \\ \notag
\end{eqnarray}
and that
\begin{eqnarray}
\left[\frac{y_{p_{k }, a_k}(x)}{y_{p_{k -1}, a_k}(x)} \right]^{4/5} & \geq &  \frac{|x - p_k|^2 + a_k^2 }{(|x  - p_k| + |p_k - p_{k - 1}|)^2 + a_k^2} \\ \notag
& \geq &\frac{1 +a_k^2/ |x - p_k|^2}{(1 + |p_k - p_{k -1}|/|x - p_k|)^2 + a_k^2/|x - p_k|^2 } \\ \notag
& \geq & \frac{1}{(1 + |p_k - p_{k - 1}| /|x - p_k|)^2} \\ \notag
& \geq& \frac{1}{\left( 1 + \gamma c_0^{-1}\tau^k \right)^2}.
\end{eqnarray}
 This then gives
\begin{eqnarray}  \label{ratio_bound}
\left[ \frac{Y_k(x)}{Y_{k -1}(x)} \right]^{4/5} & = &\left[\frac{y_{p_{k- 1}, a_k}(x)}{y_{p_{k - 1}, a_{k- 1}}(x)} \right]^{4/5} \left[ \frac{y_{p_k, a_k}(x)}{y_{p_{k - 1}, a_k}(x)} \right]^{4/5} \\ \notag 
& > &\left[ \frac{1}{1 + c_0^{-2}\gamma^{-2} \tau^{2{k- 1}}}\right]\left[\frac{1}{\left( 1 + \gamma c_0^{-1}\tau^k \right)^2} \right] \\ \notag
& = & \theta_k^{4/5}
\end{eqnarray}
where $\theta_k$ is defined by the last equality above. We also get that
\begin{eqnarray} 
q_k(x) & \geq & \frac{y_{p_k, a_k}(x)}{y_{p_{k - 1}, a_{k- 1}}(x)} q_{k -1} (x) \\ \notag
& \geq & \theta_k q_{k - 1}(x) \notag
\end{eqnarray}

 Using (\ref{ratio_bound}) above, we obtain
\begin{eqnarray} \notag
r_k(x) & = & \frac{\epsilon}{2} s_{k}^{5/3} e^{\frac{1}{2} \epsilon c_2 q_k(x)} \\ \notag
& \geq &\left[ \frac{y_{p_k, a_k}(x)}{y_{p_{k - 1}, a_{k - 1}}(x)} \right] \frac{1}{2} \epsilon y_{p_{k-1}, a_{k - 1}} (x)e^{\frac{1}{2} \epsilon c_2 \theta_kq_{k - 1}(x)} \\ \notag
& \geq & \theta_k \frac{1}{2}\epsilon s_{k- 1}^{5/3}(x)e^{\frac{1}{2} \epsilon c_2 \theta_k q_{k-1}(x)} \\ \notag
& = & \theta_k \left[ \frac{1}{2}\epsilon s_{k- 1}^{5/3}(x) \right]^{1 - \theta_k} \left[r_{k- 1}(x) \right]^{\theta^k}
\end{eqnarray}
Now, since $|x - p_{k- 1}(x)| \geq c_0 \mu^{-2/3(1 + \sigma)k}$ and $1 - \theta_k \leq c\tau^k$ for $c$ sufficiently large, we get
\begin{equation} \notag
r_k(x) >\theta_k \left[ \frac{\epsilon}{2} c_0^{5/2} \mu^{-5/3(1 + \sigma)k}\right]^{c \tau^k} \left[ r_{k-1}(x) \right ]^{\theta_k}.
\end{equation}
Now, set $c_k =\theta_k \left[ \frac{\epsilon}{2} c_0^{5/2} \mu^{-5/3(1 + \sigma)k}\right]^{c \tau^k} $. It is easily seen that $\prod_{l = 1}^{\infty}c_l > 0$. This gives the bound
\begin{equation}
r_k (x) > c_k \left(r_{k - 1}(x)\right)^{\theta_k}
\end{equation}
 and the conclusion follows by considering the separate cases $r_{k -1}(x) \geq 1$ and $r_{k - 1}(x) < 1$ (since $\theta_k < 1$). 
\end{proof}
The immediate corollary is

\begin{corollary}
Either 
\begin{equation}
r_k(x) \geq 1
\end{equation}
or
\begin{equation}
r_k(x) > \left( \prod_0^\infty c_l \right) r_0 (x).
\end{equation}
\end{corollary}
This establishes (c) of Lemma \ref{main_lemma}

\section{Proof of Theorem \ref{main_theorem}(A), (B) and (C)} 
 Note that (\ref{gauss_curvature}) and our choice of Weierstrass data gives that 
\begin{equation}
K_{\Sigma_k} (z) = \frac{-|\partial_z H_k|^2}{\text{cosh}^4 \, V_k}.
\end{equation}
For $p \in m_l$, it is clear that $F_k (p) = (0,0, p)$ for all $k$. Thus, for $k > l$ we can then estimate 
\begin{equation}
|\partial_z H_k(p)| > \frac{\mu^{-l}}{a_k^4}
\end{equation}
since $V_k(x,0) = 0$ for all $x \in \mathbb{R}$ and hence $|A_{\Sigma_k}(p)|^2 \rightarrow \infty$.
For $x \in M \setminus M_\infty$, consider the  sequence of points $p_l(x) \in m_l$. Recall that $|p_l(x)- x| < \gamma^{-l}$. We then get
\begin{equation}
|\partial_z H_{l}(p)| > \frac{\mu^{-l}}{((p - p_{l})^2 + a_{l}^2)^2} > \frac{\mu^{-l}}{(\gamma^{-2l} + a_{l}^2)^2}
\end{equation}
Taking $l \rightarrow \infty$ and  $a_l < \gamma^{-l}$ gives that $|A_{\Sigma_{l}}(p)|^2 \rightarrow \infty$ and proves (A) of Theorem \ref{main_theorem}. 

Since $V_k(x,y) > 0$ for $y > 0$, we see that $x_3(\mathbf{n}(x,y)) \neq  0$, and hence $\Sigma_k$ is graphical away from the $x_3-\text{axis}$, which proves (C) of Theorem \ref{main_theorem}.  

Now, for $\delta > 0 $ set $S_\delta = \{z | \text{dist}(\text{Re}\, z, M) < \delta \}$. From (\ref{gauss_curvature}), it is immediate that
\begin{equation}
\text{sup}_k \, \text{sup}_{\Omega_k \setminus S_{\delta}} |A_{\Sigma_k}(z)|^2 < \infty
\end{equation}
for any $\delta > 0$. This combined with Heinz's curvature estimate for minimal graphs gives (B)

\section{Proof of Theorem \ref{main_theorem} (D) and The Structure Of The Limit Lamination}

\begin{lemma} \label{convergence}
A subsequence of the embeddings $F_k: \Omega_k \rightarrow \mathbb{R}^3$ converges to a minimal lamination of $C \setminus M$
\end{lemma}
\begin{proof}
Let $K$ be a compact subset of the interior of $\Omega_\infty$. Then for $z \in K$, we have that $\text{sup}_k \, |\frac{d}{dz} H_k(z)|  < \infty$. Montel's theorem then gives a subsequence converging smoothly to a holomorphic function on $K$. By continuity of integration this gives that the embeddings $F_k: K \rightarrow \mathbb{R}^3$ converge smoothly to a limiting embedding. Thus the surfaces $\Sigma_k$ converge to a limit lamination of $C \setminus M$ that is smooth away from the $M$.
\end{proof}

Let $I = (t_1, t_2) \subset \mathbb{R}$ be an interval of the complement of the $M$ in $\mathbb{R}$ and consider $\Omega_I  = \Omega_\infty \cap \{ \text{Re} \, z \in I\}$. Then $\Omega_I$ is topologically a disk, and by Lemma \ref{convergence}, the surfaces $ \Sigma_{k,I} \equiv F_k(\Omega_I) $ are contained in $\{ t_1 < x_3 < t_2\} \subset \mathbb{R}^3$ and converge to an embedded minimal disk $\Sigma_{I}$. Now, Theorem \ref{main_theorem}(C) (which we have already established), gives that  $\Sigma_I$ consists of two multi-valued graphs $\Sigma^1_ I$, $\Sigma^2_I$ away from the $x_3-\text{axis}$.  We will show that each graph $\Sigma^j_I$ is $\infty-\text{valued}$ and spirals into the $\{x_3 = t_1\}$ and $\{x_3 = t_2 \}$ planes as claimed.

Note that by (\ref{d_yF}) and Theorem \ref{main_theorem} (C), the level sets $\{ x_3 = t\} \cap \Sigma^j_I$  for $t_1 < t < t_2$ are graphs over lines in the direction
\begin{equation}
\text{lim}_{k \rightarrow \infty} \left( \text{sin}\, U_k (t, 0), - \text{cos} \, U_k (t , 0), 0  \right)
\end{equation}
First, suppose $t_1 \in m_l$ for some $l$. Then we get that, for any $k > l$ and any $t < \frac{t_2 - t_1}{2}$.
\begin{equation}
U_k(t_1 + 2t, 0) - U_k(t_1 + t, 0)  = \int_{t_1 + t}^{t_1 + 2t} \partial_s U_k(s, 0)ds > c_2 \mu^{-l} \int_{t_1 + t}^{t_1 + 2t} \frac{ds}{((s - t_1)^2 + a^2_k)^2}
\end{equation}
(by the Cauchy-Reimann equations $U_{k,x} = V_{k, y}$). Then, since $a_k \rightarrow 0$ as $k \rightarrow \infty$, we get that
\begin{equation}
\text{lim}_{k \rightarrow \infty} U_k(t_1 + 2t, 0) - U_k(t_1 + t,0) > c_2 \mu^{-l} \int_{t_1 + t}^{t_1 + 2t} \frac{ds}{(s - t_1)^4} > \frac{c_2 \mu^{-l} }{64 t^3}
\end{equation}
and hence $\{ t_1 + t < |x_3| < t_1 + 2t \}$ contains an embedded $N_t \text{-valued graph}$, where
\begin{equation}
 N_t>\frac{c \mu^{-l} }{t^3}.
\end{equation}
Note that $N_t \rightarrow \infty$ as $t \rightarrow 0$ from above and hence $\Sigma_I$ spirals into the plane $\{ x_3 = t_1\}$. 

Now, suppose that $t_1 \notin M_\infty$. Then consider the sequence of points $p_l(t_1) \in m_l$ and recall that $t_1 - p_l(t_1) < \gamma^{-l}$. Then set $t^l = t_1 + \gamma^{-l}$ and consider the intervals
\begin{equation}
I_l = [t^{l + 1}, t^{l}].
\end{equation}
 Note that for $l$ large $I_l\subset I$. Then, for $k > l$ and $s \in I_l$ we may estimate
\begin{equation}
\partial_s U_k(s,0) > \frac{c_2 \mu^{-l}}{((s - p_{l}(t_1))^2 + a^2_k)^2} > \frac{c_2 \mu^{-l}}{(4\gamma^{-2l} + a_k^2)^2}
\end{equation}
since $s - p_{l} < 2 \gamma^{-l}$. We then get
\begin{equation}
U_k(t^j, 0) - U_k(t^{j +1}, 0) > |I_j| \frac{c_2 \mu^{-l}}{(4\gamma^{-2l} + a^2_k)^2} \geq \frac{c_2 \mu^{-l}(1 - \gamma^{-1}) \gamma^{-l}}{(4 \gamma^{-2l}+ a_k^2)^2}
\end{equation}
Taking limits, we get
\begin{equation}
\text{lim}_{k \rightarrow \infty} U_k(t^l, 0) - U_k(t^{l +1}, 0) >  \frac{c_2(1 - \gamma^{-1})}{16} \left( \frac{\gamma^3}{\mu}\right)^{l}
\end{equation}
Thus we see that $\{ t^{l + 1} < x_3 < t^{l} \} \cap \Sigma_I^l$ contains an embedded $N_l\text{-valued}$ graph, where
\begin{equation} \label{N_l}
N_l \approx  c \left( \frac{\gamma^3}{\mu}\right)^{l}
\end{equation}
This again shows that $\Sigma_I$ spirals into the plane $\{ x_3 = t_1 \}$ since as $j \rightarrow \infty$, $t^l \rightarrow t_1$, and $N_l \rightarrow \infty$.
Now for $t$ in the interior of $M$, every singly graphical component of $F_{l}$  contained in the slab $\{t - \gamma^{-l}< x_3 < t + \gamma^{-l} \}$ (by (\ref{N_l}) there are many) is graphical over $\{ x_3 = 0\} \cap B_{r_{l}}(0)$ where  Lemma \ref{good_interval} gives $r_{l} \rightarrow \infty$, which implies each component converges to  the plane $\{ x_3 = t\}$. This proves Theorem \ref{main_theorem} (D).
\section{Appendix}
Here we provide  the computations that were omitted from section \ref{preliminary_results}

\begin{proof} [proof of lemma \ref{holomorphic_on_domain}] 
It suffices to show that $h = u + iv$ is holomorphic on $\omega$. Recall that 
\begin{equation} 
h(z) = \int_0^{z} \frac{dz}{(z^2 + a^2)^2}
\end{equation} 
It is clear  that the points $\pm i a$ lie outside of $\omega$. Moreover $\omega$ is obviously simply connected, so that $ \int_0^{z} \frac{dz}{(z^2 + a^2)^2}$ gives a well-defined  holomorphic function on $\omega$.
\end{proof}

\begin{proof}[proof of lemma \ref{uv_bounds}]
We compute the real and imaginary components of $(z^2 + a^2)^2$:
\begin{equation}\notag
z^2 + a^2 = x^2 - y^2 + a^2 + 2ixy,
\end{equation}
\begin{equation} \notag
\left( z^2 + a^2 \right)^2 = \left( x^2 - y^2 + a^2 \right)^2 - 4x^2y^2 + 4ixy \left( x^2 -y^2 + a^2 \right).
\end{equation}

Set
\begin{equation}\notag
d = Re\left\{ \left( z^2 + a^2 \right)^2\right\} = \left( x^2 - y^2  + a^2 \right)^2 - 4x^2y^2,
\end{equation}
\begin{equation} \notag
b = Im \left\{ \left( z^2 + a^2\right)^2\right\}  = 4xy \left( x^2 - y^2 + a^2 \right),
\end{equation}
and
\begin{equation} \notag
c^2 = \left| \left( z^2 + a^2 \right)^2 \right|^2 = d^2 + b^2 =  \left\{ \left( x^2 - y^2  + a^2 \right)^2 - 4x^2y^2 \right\}^2 + 16x^2y^2 \left( x^2 - y^2 + a^2\right)^2.
\end{equation}
Now on $\omega$ (that is, on the set where $|y| \leq y_0(x)$), we get the bounds
\begin{equation} \notag
d \geq (1 - \epsilon^2)^2 \left( x^2 + a^2 \right)^2 - 4 \epsilon^2 (x^2 + a^2 )^2 = \left\{ (1 - \epsilon^2)^2 - 4\epsilon^2\right\} (x^2 + a^2)^2,
\end{equation}
\begin{equation} \notag
d \leq (x^2 + a^2)^2,
\end{equation}
\begin{equation} \notag
b \leq 4 \epsilon (x^2 +a^2)^{11/4} \leq 4 \epsilon (x^2 + a^2)^2
\end{equation}
since by assumption $|x|, a < \frac{1}{2}$. Using that $c^2 = d^2 + b^2$,
\begin{equation} \notag
\left\{ (1 - \epsilon^2)^2 - 4\epsilon^2 \right\}^2 (x^2 + a^2)^4  \leq c^2 \leq \left\{ 1 + 16\epsilon^2 \right\} \left( x^2 + a^2 \right)^4.
\end{equation}
Recalling that

\begin{equation} \notag
\frac{\partial}{\partial y}u(x,y) = Im \left\{ \frac{1}{\left( z^2 + a^2 \right)^2} \right\} = \frac{-b}{c^2}, \frac{\partial}{\partial y} v(x,y) = Re \left\{ \frac{1}{\left( z^2 + a^2 \right)^2} \right\} = \frac{d}{c^2}
\end{equation}
We get

\begin{equation} \notag
\left| \frac{\partial}{\partial y} u_p(x,y) \right| \leq \frac{4}{\left\{ (1 - \epsilon^2)^2 - 4\epsilon^2 \right\}^2} \frac{|x - p||y|}{\left( (x- p)^2 + a^2\right)^3}
\end{equation}
and

\begin{equation} \notag
\frac{\partial}{\partial y} v_p(x,y) \geq \frac{\left\{ (1 - \epsilon^2)^2 - 4\epsilon^2 \right\}}{1 + 16 \epsilon^2} \frac{1}{\left( (x- p)^2 + a^2\right)^2}
\end{equation}

If we restrict $\epsilon < \epsilon_0$ for $\epsilon_0$ sufficiently small we get that 
\begin{equation} \notag
\frac{4}{\left\{ (1 - \epsilon^2)^2 - 4\epsilon^2 \right\}^2} < c_1,
\end{equation}  
and

\begin{equation} \notag
\frac{\left\{ (1 - \epsilon^2)^2 - 4 \epsilon^2 \right\}}{1 + 16 \epsilon^2} > c_2.
\end{equation}
for constants $c_1$ and $c_2$. 
which immediately gives the lemma.
\end{proof}

\begin{proof}[proof of corollary \ref{U_k_V_k_bounds}]
Recalling definitions (\ref{h_l}) and (\ref{H_k}), we get

\begin{eqnarray}
\left| U_k(x, Y_k(x)) - U_k(x, 0) \right| & \leq & \sum_{l = 0}^k \mu^{-l} \int_0^{Y_k(x)} \left| \frac{\partial}{\partial y} u_l(x, y)  \right| \\ \notag
& \leq & \sum_{l = 0}^k \mu^{-l} \sum_{p \in m_l} \int_0^{Y_k(x)} \left| \frac{\partial}{\partial y} u_p(x, y)  \right| \\ \notag
& \leq & \sum_{l = 0}^k \mu^{-l} \sum_{p \in m_l} \int_0^{y_p(x)} \left| \frac{\partial}{\partial y} u_p(x, y)  \right| \\ \notag
& \leq & c_1 \epsilon^2  \sum_{l = 0}^k \mu^{-l} (\gamma^l + 1)
\end{eqnarray}

and

\begin {eqnarray}
min_{[Y_k(x)/2, Y_k(x)]} V_k(x,y) & \geq & \sum_{l = 0}^k  \int_0^{Y_k(x)/2}\frac{\partial}{\partial y} v_l(x, y) \\ \notag
& \geq & \sum_{l = 0}^k \mu^{-l} \sum_{p \in m_l} \int_0^{Y_k(x)/2}\frac{\partial}{\partial y} v_p(x,y) \\ \notag
& \geq & \frac{\epsilon c_2}{2} \sum_{l = 0}^k \mu^{-l} \sum_{p \in m_l} \frac{Y_k(x)}{y_p(x)} ((x - p)^2 + a^2)^{-3/4} \\ \notag
& = & \epsilon \frac{ c_2}{2} q_k (x)
\end{eqnarray}
\end{proof}

\end{document}